\documentclass{amsart}
\usepackage{amscd,amsmath,amssymb,amsfonts,bbm}
\usepackage[T1]{fontenc}
\usepackage{lmodern}

\usepackage{enumerate}
\usepackage{multicol}
\usepackage[all]{xy}
\usepackage{nccmath} 

\newtheorem{thm}{Theorem}[section]
\newtheorem{prop}[thm]{Proposition}
\newtheorem{lem}[thm]{Lemma}
\newtheorem{cor}[thm]{Corollary}

\theoremstyle{definition}
\newtheorem{Def}[thm]{Definition}

\theoremstyle{remark}
\newtheorem{rem}[thm]{Remark}

\numberwithin{equation}{section}

\DeclareMathOperator{\Pic}{Pic}

\DeclareMathOperator{\Sym}{Sym}

\DeclareMathOperator{\Ima}{Im}
\DeclareMathOperator{\Id}{Id}

\DeclareMathOperator{\Aut}{Aut}

\DeclareMathOperator{\Hom}{Hom}

\DeclareMathOperator{\Syz}{Syz}

\DeclareMathOperator{\HK}{HK}
\DeclareMathOperator{\length}{length}
\DeclareMathOperator{\chara}{char}

\begin{document}

\title[Complete families of smooth space curves]
{Complete families of smooth space curves and strong semistability}
\author{Olivier Benoist}
\address{Institut de Recherche Math\'ematique Avanc\'ee\\
UMR 7501, Universit\'e de Strasbourg et CNRS\\
7 rue Ren\'e Descartes\\
67000 Strasbourg, FRANCE}
\email{olivier.benoist@unistra.fr}

\renewcommand{\abstractname}{Abstract}
\begin{abstract}
We construct the first non-trivial examples of complete families of non-degenerate smooth space curves, and show that the base of such a family cannot be a rational curve.
  Both results
rely on the study of the strong semistability of certain vector bundles.
\end{abstract}
\maketitle

\section*{Introduction}\label{intro}

We work over an algebraically closed field $\mathbbm{k}$. A \textit{curve} is
a projective connected one-dimensional variety over $\mathbbm{k}$.
If $B$ is an integral variety over
$\mathbbm{k}$, a \textit{family of smooth space curves} over $B$ is a closed subvariety 
$\mathcal{C}\hookrightarrow \mathbb{P}^3_B:=\mathbb{P}^3\times B$ such that $\mathcal{C}\to B$ is a smooth family of curves.
Equivalently, it is a morphism from $B$ to the Hilbert scheme of smooth curves in $\mathbb{P}^3$.
Such a family will be said to be \textit{trivial} if all its fibers are isomorphic as subvarieties of
$\mathbb{P}^3$; in other words, if the induced morphism from $B$ to the Hilbert scheme of $\mathbb{P}^3$ is constant. It is said to be \textit{isotrivial} if all its fibers are isomorphic as abstract curves. We are interested in \textit{complete} families: those
whose base $B$ is proper.

The family of lines parametrized by the Grassmannian is a
non-trivial complete family of smooth space curves.
It is also easy to construct (ne\-ces\-sarily iso\-trivial) non-trivial complete families whose members are plane curves
\cite[Proposition 2.1]{Oolbir}. For this reason, we will restrict our attention 
to families parametrizing \textit{non-degenerate} space curves, that is curves whose linear
span is $\mathbb{P}^3$. 

Non-trivial complete
families of non-degenerate smooth space curves have been studied by Chang and Ran 
\cite{ChangRan1, ChangRan2}. They showed that the curves parametrized by the 
family can be neither rational nor elliptic \cite[Theorem 3]{ChangRan2}. They also proved that every such family comes by base-change from a family over a curve  \cite[Theorem~1]{ChangRan2},
so that one may restrict the study to this case.  

However, they do not provide examples
of such families. The existence of 
non-trivial complete families of non-degenerate smooth space curves is also stated
as an open question in \cite[p. 57]{HM}.
Our main goal is to construct examples. 

\begin{thm}\label{main}
\begin{enumerate}[(i)]
\item
There exist non-trivial complete families of non-degenerate smooth space curves over any elliptic curve.
\item
If $\mathbbm{k}$ has characteristic $p$ with $p\equiv \pm 1 [8]$,
there is a non-trivial complete family of non-degenerate smooth space curves over a smooth curve of genus $\geq 2$ that does not come by base-change from a family over a curve of genus $\leq 1$.
\end{enumerate}
\end{thm}

The curves parametrized by our families have genus $2$ and degree $5$.
As the modu\-li space of smooth curves of genus $2$ is affine \cite{Igusa},  such families are necessarily isotrivial. It is the degree $5$ line bundle providing the embedding that
varies in the family. In view of Chang and Ran's result \cite[Theorem 3]{ChangRan2}, those examples are minimal: they have both
smallest genus and smallest degree possible.

Theorem \ref{main} (i) is also optimal in the sense that the genus of the base is minimal:

\begin{thm}
\label{main2}
There are no non-trivial complete families of non-degenerate smooth space curves over $\mathbb{P}^1$.
\end{thm}

Theorem \ref{main} shows the existence of elliptic curves (and, when $\chara(\mathbbm{k})\equiv \pm 1[8]$,
of a curve of genus $\geq 2$) in the Hilbert scheme
of non-degenerate smooth space curves: it fits into the classical theme of constructing complete subvarieties of moduli spaces
initiated by Oort \cite{Oortcomplete}.

We do not know how to remove the hypothesis on $\mathbbm{k}$ in Theorem \ref{main} (ii) (see Remark \ref{remimprove}).  We also leave open the question whether there exist non-isotrivial complete families of smooth space curves. Since there do not exist
non-isotrivial complete families of smooth curves over curves of genus $\leq 1$ \cite[Th\'eor\`eme 4]{Szpiro}, Theorem \ref{main} (ii) may be viewed as a first step towards constructing 
non-isotrivial complete families of smooth space curves. 

\vspace{1em}

  In section \ref{section1}, we study when an abstract family of smooth polarized curves over a smooth projective curve gives rise to a non-trivial family of non-degenerate smooth space curves. We obtain necessary conditions in Proposition \ref{mainth2bis} and sufficient conditions in Proposition \ref{emb} that yield proofs of Theorem \ref{main2} and Theorem \ref{main}, respectively.
 A key role is played by the strong semistability of some vector bundles on the base, and \S\ref{sss} is devoted to recalling generalities on strong semistability. 

  The proof of Theorem \ref{main} (ii) in section \ref{section1} requires to verify the strong semistability of some vector bundles. We postpone this important step to section~\ref{secredmodp}. Our strategy there is to ensure that the relevant bundles are syzygy bundles (see Definition \ref{defsyz}). The strong semistability of such bundles has been related by Brenner \cite{Brenner} and Trivedi \cite{Trivedi} to Hilbert-Kunz multiplicities (see Definition \ref{defHK} and Theorem \ref{BrennerTrivedi}).
 In our situation, we do not know how to compute the relevant Hilbert-Kunz multiplicities directly, as Han and Monsky did for Fermat curves \cite{Han, HanMonsky, Monsky}. Instead, we take inspiration from \cite{BrennerKaid}, where Brenner and Kaid obtain  stronger results than strong semistability (explicit Frobenius periodicity up to a twist) for some syzygy bundles over Fermat curves.
The strategy of \cite{BrennerKaid} uses crucially the semistabi\-lity of the syzygy bundles, that is known thanks to Han and Monsky. We need to replace these arguments by different ones: explicit syzygy computations using the strong Lefschetz property of 
appropriate homogeneous ideals (see \S\ref{parsyz}). A benefit of our method is that it allows us to give new examples of how 
Hilbert-Kunz multiplicities vary with the characteristic of the base field in Theorem \ref{thHK}.

\bigskip

{\it Acknowledgements.} I would like to thank Chungsim Han for having kindly made available to me a copy of her thesis \cite{Han}.

\section{Embedding abstract families}
\label{section1}

\subsection{Strong semistability} \label{sss}
If $\mathbbm{k}$ is of positive characteristic, and $X$ is a variety over $\mathbbm{k}$,
we denote by $F:X\to X$ the absolute Frobenius morphism.

\begin{Def}
A vector bundle $\mathcal{E}$ on a smooth curve $B$ is \textit{strongly semistable} if $\mathbbm{k}$ is of characteristic $0$ and 
$\mathcal{E}$ is semistable, or if $\mathbbm{k}$ is of positive characteristic and
for every $k\geq 0$, $F^{k*}\mathcal{E}$ is semistable. 
\end{Def}
Unlike semistability, strong semistability is always preserved by finite base-change, tensor products and symmetric powers \cite[2.2.2, 2.2.3]{Langersurvey}.
The following important theorem is due to Langer \cite[Theorem 2.7]{Langer}:

\begin{thm}\label{HNfort}
Let $\mathcal{E}$ be a vector bundle on a smooth curve $B$. Then there exists a finite morphism from a smooth curve $f:B'\to B$
such that the graded pieces of the Harder-Narasimhan filtration of $f^*\mathcal{E}$ are strongly semistable.
 \end{thm}

Such a filtration will be called a \textit{strong Harder-Narasimhan filtration}.
In charac\-teristic $0$, the Harder-Narasimhan filtration is always strong.
Over elliptic curves, the situation is very simple:

\begin{prop}\label{ell} Let $\mathcal{E}$ be an indecomposable vector bundle over an elliptic curve.
\begin{enumerate}[(i)]
\item $\mathcal{E}$ is strongly semistable,
\item $\mathcal{E}$ is stable if and only if its degree is prime to its rank.
\end{enumerate}
\end{prop}
\begin{proof}
In the first statement, the semistability of $\mathcal{E}$ is proved in \cite[Lemma 1]{Ploog}.
The strong semistability then follows from the more general \cite[Theorem 2.1]{MehtaRama}. 

A semistable vector bundle whose rank and degree are prime to each other is clearly stable. Conversely, when the degree and the rank
of $\mathcal{E}$ are not prime to each other, Oda has proved \cite[Corollary 2.5]{Oda} that $\mathcal{E}$ is not simple, hence not stable.
\end{proof}

We will need conditions ensuring that a vector bundle becomes isomorphic
to a direct sum of isomorphic line bundles after an appropriate base-change. This is the goal of the two following propositions.
The first one might be well-known, but I do not know a reference for it. The second one is the Lange-Stuhler theorem.

\begin{prop}\label{sssell}
Let $\mathcal{E}$ be a stable vector bundle over an elliptic curve $E$.
Then there exists an
isogeny $f:E'\to E$ such that $f^*\mathcal{E}$ is isomorphic
to a direct sum of isomorphic line bundles.
\end{prop}

\begin{proof}
By Proposition \ref{ell} (i), the pull-back of $\mathcal{E}$ by any isogeny is semistable. 

We first claim that there exists an isogeny $f:E'\to E$ such that $f^*\mathcal{E}$ is isomorphic
to a direct sum of line bundles. To prove it, let $f:E'\to E$ be an isogeny whose degree is divisible by the rank of $\mathcal{E}$.
Write $f^*\mathcal{E}$ as a direct sum
of indecomposable bundles. If $\mathbbm{k}$ is of characteristic $0$, those indecomposable bundles are all stable of the same slope
by \cite[Lemma 3.2.3]{Huybrechts}, and Proposition \ref{ell} (ii) implies that they are line bundles.
If $\mathbbm{k}$ is of positive characteristic $p$, Proposition \ref{ell} (ii)
shows that $f^*\mathcal{E}$ cannot be stable. Considering a Jordan-H\"older filtration for $f^*\mathcal{E}$ and using induction on the rank of $\mathcal{E}$, it is possible to suppose that all the graded pieces of this filtration have rank $1$.
Now, extensions between line bundles of
the same degree are trivial if the line bundles are not isomorphic, and
parametrized by $H^1(E',\mathcal{O}_{E'})$ otherwise.
Let $[p]:E'\to E'$ denote the multiplication by $p$ isogeny. Since the dual of $[p]$ (that is $[p]$ itself) is  not separable, the pull-back map $[p]^*:H^1(E',\mathcal{O}_{E'})\to H^1(E',\mathcal{O}_{E'})$ vanishes.
Base-changing by an appropriate power of $[p]$ thus splits all extensions appearing in the Jordan-H\" older filtration, and proves our claim.

It remains to prove that, up to another base-change by an isogeny, all these line bundles are isomorphic.
 Let us write $f^*\mathcal{E}\simeq\bigoplus_i\mathcal{F}_i$, where the $\mathcal{F}_i$ are the isotypical factors:
each $\mathcal{F}_i$ is the direct sum of isomorphic line bundles.
Write $f=g\circ h$, where $h:E'\to F$ is separable of Galois group $G$ and $g:F\to E$ is purely inseparable. If the group $G$ did not
act transitively on the isotypical factors, a non-trivial direct sum $\mathcal{G}$ of some of them would descend to $F$ by Galois descent.
Since, $\Hom(\mathcal{G}, f^*\mathcal{E}/\mathcal{G}\otimes \Omega^1_{E'})=\Hom(\mathcal{G}, f^*\mathcal{E}/\mathcal{G})=0$,
inseparable descent \cite[Theorem 5.1]{descinsep} shows that this sheaf descends even to $E$,
contradicting the stability of $\mathcal{E}$. Hence $G$ permutes transitively the isotypical components. But since
$G$ acts on $E'$ as a finite subgroup of translations, it follows that the line bundles appearing in $\mathcal{E}$ differ from each other by
torsion line bundles. Hence all the line bundles appearing become isomorphic after further pull-back by a well-chosen isogeny.
\end{proof}

\begin{prop}\label{sssfini}
Let $\mathcal{E}$ be a vector bundle on a smooth curve $B$ over the algebraic closure of a finite field. Then the following
conditions are equivalent:
\begin{enumerate}[(i)]
\item $\mathcal{E}$ is strongly semistable.
\item There exists a finite morphism from a smooth curve $f:B'\to B$ such that $f^*\mathcal{E}$ is isomorphic to a direct sum of isomorphic line bundles.
\end{enumerate}
\end{prop}

\begin{proof}
If (ii) holds, the vector bundle $f^*F^{k*}\mathcal{E}=F^{k*}f^*\mathcal{E}$ is semistable as a direct sum of isomorphic line bundles. This implies that $F^{k*}\mathcal{E}$ is semistable, proving (i).

Let us explain the other implication, due to Lange and Stuhler \cite{LangeStuhler}.
First, it is easy to find a finite morphism from a smooth curve $g:B''\to B$ and a line bundle $\mathcal{N}$ on $B''$ such that $g^*\mathcal{E}\otimes \mathcal{N}$
has degree $0$. By our hypothesis on the base field, the strongly semistable vector bundle $g^*\mathcal{E}\otimes \mathcal{N}$ is trivialized by a finite surjective morphism
$h:B'\to B''$ by \cite[Satz 1.9]{LangeStuhler}. Setting $f=g\circ h$, one sees that $f^*\mathcal{E}$ is a direct sum of line bundles isomorphic to $h^*\mathcal{N}^{-1}$.
\end{proof}

\subsection{The Harder-Narasimhan filtration} We start with a lemma:

\begin{lem}\label{lemE}
Let $B$ be a smooth curve, let $\pi:\mathcal{C}\to B$ be a smooth projective family of curves over $B$ and let $\mathcal{L}$ be a
line bundle on $\mathcal{C}$. 
Then $\mathcal{E}:=\pi_*\mathcal{L}$ is locally free and its formation commutes with base-change by any finite map from a smooth curve $B'\to B$.
Moreover, for every $b\in B$, the natural map $\mathcal{E}|_b\to H^0(\mathcal{C}_b,\mathcal{L}_b)$ is injective.
\end{lem}

\begin{proof}
The sheaf $\mathcal{E}$ is locally free as a torsion-free coherent sheaf over a smooth curve.
The second statement is a consequence of flat base-change \cite[III Proposition 9.3]{Hartshorne}.
As for the third statement, consider the exact sequence $0\to \mathcal{O}_B(-b)\to\mathcal{O}_B\to\mathcal{O}_B|_b\to 0$.
Pull it back to $\mathcal{C}$, tensor with $\mathcal{L}$ and push it forward to $B$ to get an exact sequence
$0\to\mathcal{E}(-b)\to\mathcal{E}\to\Ima(\mathcal{E}\to H^0(\mathcal{C}_b,\mathcal{L}_b))\to 0$.
Restricting it to $b$ using right-exactness of tensor product, and noticing that the morphism $\mathcal{E}(-b)|_b\to\mathcal{E}|_b$ vanishes,
one sees that $\mathcal{E}|_b\to H^0(\mathcal{C}_b,\mathcal{L}_b)$ is indeed injective.
\end{proof}

In the following proposition, we make use of the \textit{secant variety} $S\subset\mathbb{P}^4$ of a smooth curve $C\subset\mathbb{P}^4$, which is the union of all lines in $\mathbb{P}^4$ that meet $C$ with multiplicity $\geq 2$.

\begin{prop}\label{mainth2bis}
Let $\pi:\mathcal{C}\to B$, $\phi:\mathcal{C}\hookrightarrow\mathbb{P}^3_B$ be a non-trivial complete family of non-degenerate smooth space curves over
a smooth curve $B$, and $\mathcal{E}:=\pi_*\phi^*\mathcal{O}_{\mathbb{P}^3}(1)$. 
Then the constant subbundle
$\mathcal{O}_B^{\oplus 4}\subset\mathcal{E}$ with fibers $H^0(\mathbb{P}^3,\mathcal{O}_{\mathbb{P}^3}(1))$
is the first step of the Harder-Narasimhan filtration of $\mathcal{E}$.
\end{prop}

\begin{proof}
We argue by contradiction and suppose that the conclusion does not hold. The idea of the proof is to use the hypothesis that $\mathcal{O}_B^{\oplus 4}\subset\mathcal{E}$
is not the first step of the Harder-Narasimhan filtration of $\mathcal{E}$ to produce an embedding of $\mathcal{C}$ in a four-dimensional projective bundle over $B$, and to derive a contradiction by studying geometrically this embedding.
At any point of the proof, we may replace $B$ by a finite cover by a smooth curve $B'$ because
the formation of $\mathcal{E}$ commutes with this base-change by Lemma \ref{lemE}. 

Let $\mathcal{Q}$ be the quotient of $\mathcal{E}$ by $\mathcal{O}_B^{\oplus 4}$.
Using Theorem \ref{HNfort}, perform a base-change
to ensure that the strong Harder-Narasimhan filtration of $\mathcal{Q}$ is defined over $B$. Since $\mathcal{O}_B^{\oplus 4}$ is not the first step of the Harder-Narasimhan filtration of $\mathcal{E}$, the first step of the Harder-Narasimhan filtration of $\mathcal{Q}$ has non-negative degree. 

Choosing a field of definition of
finite type, spreading out and specializing to a general closed point, we get data defined over a finite field. It still contradicts the proposition, as $\mathcal{Q}$ still has a subbundle of non-negative degree after such a general specialization.
Consequently, we may suppose
that $\mathbbm{k}$ is the algebraic closure of a finite field. As above,
we may assume that the strong Harder-Narasimhan filtration of $\mathcal{Q}$ is defined over $B$ and that its
first step has non-negative degree.

Base-changing again using Proposition \ref{sssfini}, we may assume that this subbundle is a direct sum of line bundles of non-negative degree.
In particular, $\mathcal{Q}$ contains a subbundle $\mathcal{M}$ of rank $1$ of non-negative degree.
Consequently, there exists a subbundle $\mathcal{F}$ of $\mathcal{E}$ that is an extension of a line bundle of non-negative
degree $\mathcal{M}$ by $\mathcal{O}_B^{\oplus 4}$.
Base-changing using Theorem \ref{HNfort}, the strong Harder-Narasimhan filtration of $\mathcal{F}$ is defined over $B$. Let $\mathcal{G}\subset\mathcal{F}$ be the first step of this filtration.

Now, let us use $\mathcal{F}$ to embed $\mathcal{C}$ in a relative projective bundle over $B$: we get an immersion $\psi:\mathcal{C}\to\mathbb{P}_B\mathcal{F}$. Moreover, one recovers the original embedding $\phi$ by projecting away from $\mathbb{P}_B\mathcal{M}$. Note that, as $\mathcal{F}|_b\to H^0(\mathcal{C}_b,\mathcal{L}_b)$
is injective by Lemma \ref{lemE}, $\psi$ embeds all fibers of $\pi$ in a non-degenerate way in $\mathbb{P}^4$. Let us introduce the relative secant variety  $\mathcal{S}\hookrightarrow\mathbb{P}_B\mathcal{F}$ that is the union of the secant varieties of the embedded
curves $\mathcal{C}_b\hookrightarrow \mathbb{P}\mathcal{F}_b$. It is a
hypersurface of $\mathbb{P}_B\mathcal{F}$ because secant varieties of non-degenerate curves in $\mathbb{P}^4$ are of dimension $3$. It does not meet
$\mathbb{P}_B\mathcal{M}$ because, for every $b\in B$, the linear system $H^0(\mathbb{P}^3,\mathcal{O}_{\mathbb{P}^3}(1))$ induced an embedding of $\mathcal{C}_b$.

 Let $q:\mathbb{P}_B\mathcal{F}\to B$ be the projection and
$\mathcal{O}_q(1)$ be the relative tautological bundle.
 By description of the Picard group of a projective bundle, there exist $\mathcal{A}\in\Pic(B)$ and $l\in\mathbb{Z}$
such that $\mathcal{S}$ is the zero-locus of a section $\sigma\in H^0(\mathbb{P}_B\mathcal{F},\mathcal{O}_q(l)\otimes q^*\mathcal{A})=
H^0(B,\Sym^l\mathcal{F}\otimes \mathcal{A})$. That $\mathcal{S}$ does not meet $\mathbb{P}_B\mathcal{M}$
means exactly that $\sigma$ induces a nowhere vanishing section of $\mathcal{M}^{\otimes l}\otimes \mathcal{A}$ on $B$. In particular,
$\mathcal{A}\simeq \mathcal{M}^{ \otimes -l}$.

We distinguish three cases. Suppose first
that $\mu(\mathcal{G})<\mu(\mathcal{M})$, so that the graded pieces $\mathcal{G}_i$ of the strong Harder-Narasimhan filtration
 of $\mathcal{F}$ all have slope $<\mu(\mathcal{M})$.
 This filtration induces a filtration of $\Sym^l\mathcal{F}$ whose graded pieces are tensor products
of symmetric powers of the $\mathcal{G}_i$: these are strongly semistable of slope $<\mu(\mathcal{M}^{\otimes l})$.
Consequently, $H^0(B,\Sym^l\mathcal{F}\otimes \mathcal{M}^{\otimes -l})=0$, which is a contradiction.

Next, suppose that $\mu(\mathcal{G})\geq\mu(\mathcal{M})>0$.
The morphism $\mathcal{G}\to\mathcal{M}$ cannot be zero as there are no non-zero morphisms
$\mathcal{G}\to \mathcal{O}_B^{\oplus 4}$ by semistability of $\mathcal{G}$.
Again by semistability of $\mathcal{G}$, this morphism has to be surjective. Then $\mathcal{G}$ is an extension of $\mathcal{M}$
by a subbundle of $ \mathcal{O}_B^{\oplus 4}$, and the inequality $\mu(\mathcal{G})\geq\mu(\mathcal{M})$ implies that
$\mathcal{G}\to \mathcal{M}$ is an isomorphism. Hence $\mathcal{F}$ splits as a direct sum $\mathcal{O}_B^{\oplus 4}\oplus \mathcal{M}$.
The space  $H^0(B,\Sym^l\mathcal{F}\otimes \mathcal{M}^{\otimes -l})$ is one-dimensional because $\mu(\mathcal{M})>0$, and the zero locus of one of
its sections on a fiber of $q$ is a hyperplane with multiplicity $l$. This contradicts the fact that, the curve $\mathcal{C}_b$ being embedded
in a non-degenerate way in $\mathbb{P}^4$, its secant variety is also non-degenerate.

Finally, suppose that $\mu(\mathcal{M})=0$.
Then $\mathcal{F}$ is strongly semistable as an extension of strongly semistable bundles of the same degree.
Applying Proposition \ref{sssfini}, we may assume that $\mathcal{F}$ is a direct sum of isomorphic
line bundles, so that $\mathbb{P}_B\mathcal{F}\simeq \mathbb{P}^4_B$. The relative secant variety $\mathcal{S}$ is then a hypersurface of $\mathbb{P}^4_B$
avoiding a constant section. It follows that $\mathcal{S}$ is a product hypersurface, isomorphic to $S\times B$ where $S\subset\mathbb{P}^4$ is a hypersurface.
Consequently, $S$ is the secant variety of all curves $\mathcal{C}_b\hookrightarrow \mathbb{P}^4$.
Recall that Chang and Ran \cite[Theorem~3]{ChangRan2} proved that the curves
$\mathcal{C}_b$ have genus $\geq 2$. 
Hence, by Lemma \ref{secant} below, there are only finitely many possibilities for the curves $\mathcal{C}_b\subset\mathbb{P}^4$, and the subvariety $\psi:\mathcal{C}\hookrightarrow \mathbb{P}^4_B$ has to be a product itself.
Since the original family $\phi$ is obtained by projecting away from a constant section,
it follows that our original family was a product, contradicting its non-triviality.
\end{proof}

We have used the following lemma:

\begin{lem}\label{secant}
Let $C\subset \mathbb{P}^4$ be a smooth non-degenerate curve of genus at least $2$, and let $S$ be its secant variety. Then there is a unique family of lines
that covers $S$, namely the $2$-dimensional family of secants of $C$. Moreover, $C$ is an irreducible component of the set of points included in
infinitely many of these lines. 
\end{lem}

\begin{proof}
Let $P$ be the $\mathbb{P}^1$-bundle over the two-fold symmetric product $C^{(2)}$ of $C$ whose fiber over $(x,y)\in C^{(2)}$ is the line through $x$ and $y$ if $x\neq y$ (resp. the tangent at $x$ if $x=y$).
The natural surjective morphism $p:P\to S$ is birational by  \cite{Dale}.
As this claim is not explicitly stated by Dale, we explain how to deduce it from \cite{Dale}.

 To do so, we introduce a few notation. Let $Q$ be the pull-back of $P$ by the degree $2$ morphism $C^2\to C^{(2)}$: it is a $\mathbb{P}^1$-bundle over $C^2$.
Define $M(C)$ to be the set of triples $(x,z,l)$, where $x\in C$, $z\in \mathbb{P}^4$, and $l$ is a line containing $x$ and $z$ that is secant to $C$. The morphism $r:Q\to M(C)$ sending a point $z$ on the line $l$ over $(x,y)$ to $(x,z,l)$ is birational by \cite[Theorem 1.8]{Dale}. Define $SB(C)$ to be the set of pairs $(z,l)$ where $l$ is a line secant to $C$ and $z\in l$. The morphism $s:M(C)\to SB(C)$ defined by $(x,z,l)\mapsto (z,l)$ is separable of degree $2$ by \cite[Theorem 1.8, Lemma 3.5]{Dale}, and the morphism $t:SB(C)\to S$ defined by $t(z,l)=z$ is birational  by \cite[Theorem 4.1, Theorem 1.10]{Dale}. The composition $t\circ s\circ r:Q\to S$ then has degree $2$. Since it factors as the composition of the natural degree $2$ morphism $Q\to P$ and of $p:P\to S$, it follows that $p$ is indeed birational.

Since $C$ has genus $\geq 2$, the Abel-Jacobi map shows that $C^{(2)}$
contains only finitely many rational curves. Hence, the only family of rational curves that covers
$S$ is the one induced by the fibers of the $\mathbb{P}^1$-bundle structure, that is the family of secants of $C$.
The subset of $S$ consisting of points included in infinitely many of these secants is the image by $t$ of the union of the positive-dimensional fibers of $t$. 
A dimension count  shows that it is an algebraic variety of dimension  at most $1$.
Since $C$ is obviously contained in it, $C$ has to be an irreducible component of this locus. 
\end{proof}

Proposition \ref{mainth2bis} gives necessary conditions for a polarized family $(\pi:\mathcal{C}\to B,\mathcal{L})$ to induce a non-trivial family of non-degenerate smooth space curves $\phi:\mathcal{C}\hookrightarrow\mathbb{P}^3_{B}$ with $\mathcal{L}\simeq\phi^*\mathcal{O}_{\mathbb{P}^3}(1)$:
the first graded piece of the Harder-Narasimhan filtration of $\mathcal{E}:=\pi_*\mathcal{L}$ has to
be of rank $4$, and the corresponding sections have to induce embeddings of the fibers of $\pi$ in $\mathbb{P}^3$.
The proof of Theorem \ref{main2} follows:

\begin{proof}[Proof of Theorem \ref{main2}]
Let $\pi:\mathcal{C}\to \mathbb{P}^1$, $\phi:\mathcal{C}\hookrightarrow\mathbb{P}^3_{\mathbb{P}^1}$ be a
complete family of non-degenerate smooth space curves over
$\mathbb{P}^1$.
It is isotrivial by \cite[Th\'eor\`eme 4]{Szpiro}:
all the fibers of $\pi$ are isomorphic to a fixed curve $C$. By Chang and Ran \cite[Theorem 3]{ChangRan2}, $C$ has genus
$\geq 2$. Since the automorphism group of $C$ is finite \cite{Schmid} and
$\mathbb{P}^1$ is simply connected, the family has to be a product: $\mathcal{C}\simeq C\times \mathbb{P}^1$.

Since the Picard scheme $\Pic(C)$ does
not contain non-trivial rational curves, all the fibers are even isomorphic as polarized curves
and $\phi^*\mathcal{O}_{\mathbb{P}^3}(1)\simeq \mathcal{M}\boxtimes \mathcal{N}$ for some line bundles $\mathcal{M}$
(resp. $\mathcal{N}$) on $C$ (resp. $\mathbb{P}^1$).
Consequently, $\mathcal{E}:=\pi_*\phi^*\mathcal{O}_{\mathbb{P}^3}(1)$ is isomorphic to a direct sum of isomorphic line bundles, hence is strongly semistable.
It follows from Proposition \ref{mainth2bis} that the subbundle of $\mathcal{E}$ used to construct the embedding $\phi$ is $\mathcal{E}$ itself,
so that our family is trivial.
\end{proof}

\subsection{Constructing embeddings} We now provide a sufficient condition for an abstract family of curves to give rise to a complete family of non-degenerate smooth space curves, up to maybe replacing the base by a finite surjective cover.

\begin{prop}\label{emb}
Let $\pi :\mathcal{C}\to B$  be a smooth projective family of curves over a smooth projective curve.
Let $\mathcal{L}$ be a line bundle on $\mathcal{C}$ and $\mathcal{E}:=\pi_*\mathcal{L}$.
Let $\mathcal{F}\subset\mathcal{E}$ be a subbundle of rank $4$
such that for every $b\in B$, $\mathcal{F}|_b\subset H^0(\mathcal{C}_b,\mathcal{L}_b)$ embeds $\mathcal{C}_b$ in $\mathbb{P}^3$.

Suppose that one of the following conditions is satisfied:
\begin{enumerate}[(i)]
\item $\mathcal{F}$ is stable and $B$ is an elliptic curve,
\item $\mathcal{F}$ is strongly semistable and $\mathbbm{k}$ is the algebraic closure of a finite field.
\end{enumerate}
Then there exist a finite morphism from a smooth curve $f:B'\to B$  
and, denoting by $(\pi':\mathcal{C}'\to B', \mathcal{L}')$ the base-change, a complete family of non-degenerate smooth space curves $\phi':\mathcal{C}'\hookrightarrow\mathbb{P}^3_{B'}$ 
such that $\mathcal{L}'|_{\mathcal{C}'_b}\simeq\phi'^*\mathcal{O}(1)|_{\mathcal{C}'_b}$ for every $b\in B'$.

\nopagebreak[1]
Moreover, in case (i), $B'$ may be chosen isomorphic to $B$.
\end{prop}

\begin{proof}
By Propositions \ref{sssell} and \ref{sssfini}, there exists a finite morphism from a smooth curve $f:B'\to B$
such that $f^*\mathcal{F}$ is isomorphic to a direct sum of isomorphic line bundles.
By Lemma \ref{lemE}, $f^*\mathcal{F}$ is a subbundle of
$\pi'_*\mathcal{L}'$, and for every $b\in B$, $(f^*\mathcal{F})|_b\subset H^0(\mathcal{C}'_b,\mathcal{L}'_b)$
embeds $\mathcal{C}'_b$ in a non-degenerate way in $\mathbb{P}^3$. Consequently, $f^*\mathcal{F}$ induces
an embedding $\phi':\mathcal{C}'\hookrightarrow\mathbb{P}_{B'}(f^*\mathcal{F})$ that is non-degenerate over every $b\in B'$. Since this
projective bundle is trivial by our choice of $f$, we are done.

In case (i), one may choose $f$ to be an isogeny by Proposition \ref{sssell}. The isogeny $f$ factors some multiplication isogeny $[N]:B\to B$, allowing us to assume $B'=B$.
\end{proof}

\begin{rem}\label{remnt}
A family constructed by Proposition \ref{emb} is non-trivial if the $(\mathcal{C}_b,\mathcal{L}_b)$ are not all isomorphic as
polarized curves. In this case, Proposition \ref{mainth2bis} shows that $\mathcal{F}$ has to be the first graded piece of
the Harder-Narasimhan filtration of $\mathcal{E}$.
\end{rem}

\begin{rem}
In Proposition \ref{emb} (ii), the genus of the curve $B'$ is not explicit as the construction of $B'$
relies on the Lange-Stuhler theorem (Proposition \ref{sssfini}).
However, it follows from the proof of this theorem
\cite[Satz 1.4 b)]{LangeStuhler} that we can give bounds for the genus of $B'$
if we know an explicit Frobenius periodicity property for the strongly semistable vector bundle $\mathcal{F}$
(that is a relation of the form $F^{r*}\mathcal{F}\simeq F^{s*}\mathcal{F}\otimes \mathcal{N}$ for some $r\neq s$ and some line bundle $\mathcal{N}$).

 Fortunately, in our applications to Theorem \ref{main} (ii), we prove the strong
 semistability of the relevant vector bundle precisely by exhibiting such a relation
(see the proof of Corollary \ref{coross}). Consequently, Lange-Stuhler's proof provides bounds for the genus of
the base of the families constructed in Theorem \ref{main} (ii).
\end{rem}

\subsection{Curves of genus $2$ and degree $5$}\label{par25}
We may now give the:

\begin{proof}[Proof of Theorem \ref{main}]
Let $C$ be a smooth curve of genus $2$ and $\mathcal{L}$ be a degree $5$ line bundle on $C$. The Riemann-Roch theorem
shows that $h^1(C,\mathcal{L})=0$, $h^0(C,\mathcal{L})=4$, and that these four sections embed $C$ in $\mathbb{P}^3$. 
Let $A:=\Pic^5(C)$ be the variety parametrizing degree $5$ line bundles on $C$ and $\mathcal{P}$ be a Poincar\'e bundle on $C\times A$. 
By \cite[III Theorem 12.11]{Hartshorne}, the sheaf $\mathcal{E}:=p_{2*}\mathcal{P}$
is a rank $4$ vector bundle on $A$ whose formation commutes with base-change.

Let $B$ be a smooth projective curve and $i:B\to A$ be a non-constant morphism. Consider the constant family $\pi:\mathcal{C}:=C\times B\to B$ polarized by
$\mathcal{L}:=(\Id,i)^*\mathcal{P}$. By base-change, $\pi_*\mathcal{L}=i^*\mathcal{E}$. 
To prove Theorem \ref{main}, we apply Proposition \ref{emb} to polarized families
$(\mathcal{C}\to B,\mathcal{L})$ as above, with $\mathcal{F}:=i^*\mathcal{E}$.

Let us show that it is possible to choose $C$, $B$ and $i$ carefully so that the stability hypotheses (i) or (ii) in Proposition \ref{emb} are satisfied.
In the setting of Theorem \ref{main} (i), the curve $B$ be an elliptic curve, and Proposition \ref{Propjacnonsimple} below produces a genus $2$ curve $C$ and a non-constant morphism $i:B\to A$ such that $i^*\mathcal{E}$ is stable.  
In the setting of Theorem \ref{main} (ii), the field $\mathbbm{k}$ is of characteristic $p\equiv\pm 1[8]$, and Proposition \ref{ex2} proven in section \ref{secredmodp} produces a genus~$2$ curve $C$ 
and, setting $B:=C$, an immersion $i:B\to A$, both defined over $\bar{\mathbb{F}}_p$,
such that $i^*\mathcal{E}$ is strongly semistable. 

To conclude the proof, it remains to verify that the families of smooth space curves constructed by applying Proposition \ref{emb} are non-trivial. To do so, fix $b\in B$, and consider the polarized variety $(\mathcal{C}_b,\mathcal{L}_b)$. Since $i$ is non-constant and $\Aut(\mathcal{C}_b)$ is finite \cite{Schmid}, there are at most finitely many $b'\in B$ such that $(\mathcal{C}_b,\mathcal{L}_b)\simeq(\mathcal{C}_{b'},\mathcal{L}_{b'})$, allowing to apply Remark \ref{remnt}.
\end{proof}

\begin{rem}
\label{remimprove}
The difficulty of removing the assumption that $p\equiv\pm 1[8]$ in the statement of Theorem \ref{main} (ii) lies in the verification of the strong semistability assumption in Proposition \ref{emb} (ii), for an appropriate choice of $C$, $B$ and $i$. 
\end{rem}

Proposition \ref{Propjacnonsimple} relies on a construction of curves of genus $2$ whose jacobian is not simple, that is very well explained in
the first section of \cite{FreyKani}. We keep the notations $A$ and $\mathcal{E}$ of the proof of Theorem \ref{main} above.

\begin{prop}\label{Propjacnonsimple}
Let $B$ be an elliptic curve over $\mathbbm{k}$. Then there exist a genus $2$ curve $C$ and an immersion $i:B\to  A:=\Pic^5(C)$ such that $i^*\mathcal{E}$ is stable.
\end{prop}

\begin{proof}
Let $E$ be an elliptic curve over $\mathbbm{k}$ not isogenous to $B$. Let $n$ be an odd integer invertible in $\mathbbm{k}$. 
Choose an isomorphism $E[n]\stackrel{\sim}\longrightarrow B[n]$ whose graph $\Gamma$ is isotropic with respect to the Weil pairings on $E[n]$ and $B[n]$. 
Let $A:=(E\times B)/\Gamma$. The quotient of $A$ by the image $B$ of
$\{0\}\times B$ in $A$ is $(E\times B)/\langle \{0\}\times B,\Gamma\rangle\simeq E/E[n]\simeq E,$ yielding an
exact sequence $0\to B\to A\stackrel{q}\longrightarrow E\to 0$ of abelian varieties.
By \cite[Propositions 1.1 and 1.4]{FreyKani}, $A$ is isomorphic to the Jacobian of a smooth curve $C$ of genus $2$ and the
theta divisor of $A$ has degree $n$ on $B$. 

Choose an isomorphism $A\simeq \Pic^5(C)$ and let $i:B\to A$ be the inclusion of a general fiber of $q$. Suppose for contradiction that $i^*\mathcal{E}$ is not stable.
As $\det(\mathcal{E})$ is numerically equivalent to the opposite of the theta divisor  \cite[VII.4]{ACGH1}, the rank $4$ of $i^*\mathcal{E}$ is prime with its degree $-n$, showing that $i^*\mathcal{E}$ is not semistable.

By the existence of a relative Harder-Narasimhan filtration with respect to $q$ \cite[Theorem 2.3.2]{Huybrechts}, there exists a saturated subsheaf $\mathcal{F}\subset \mathcal{E}$ whose restriction to a general fiber
of $q$ destabilizes $\mathcal{E}$. Outside of a finite number of points of $A$, $\mathcal{F}$ is a vector bundle. Its determinant $\det(\mathcal{F})$
extends uniquely as a line bundle $\mathcal{N}$ on $A$ by smoothness of $A$.
By construction, $\mathcal{N}$ has degree $>-n$ on the fibers of $\mathcal{F}$. 

Consider the projection $u:E\times B\to A$. The isomorphism class of the line bundle $u^*\mathcal{N}$ is $\Gamma$-invariant. 
Since $E$ and $B$
are not isogenous, $\Pic(E\times B)\simeq \Pic(E)\oplus \Pic(B)$. The action of $\Gamma$ on $\Pic(E\times B)$ is easy to describe,
and one sees that $\Pic(E\times B)^{\Gamma}$ consists of line bundles of the form $\mathcal{N}_E\boxtimes \mathcal{N}_B$, where
$\mathcal{N}_E$ (resp. $\mathcal{N}_B$) have degree divisible by $n$ on $E$ (resp. $B$). Hence $\mathcal{N}\cdot B=u^*\mathcal{N}\cdot (\{0\}\times B)$ is a multiple of $n$.

Hence, the restriction of $\mathcal{F}$ to a general fiber of $q$ has non-negative degree. Equi\-valently, the restriction of $\mathcal{F}^\vee$
to a general fiber of $q$ has non-positive degree. Consequently, the vector bundle $\mathcal{E}^\vee$ is not ample.
This contradicts \cite[VII 2.2]{ACGH1} .
\end{proof}

\section{Constructing strongly semistable vector bundles}
\label{secredmodp}
In this section, $\mathbbm{k}$ is assumed to be of positive characteristic $p$.

Let $C$ be a smooth curve of genus $2$, $c\in C$ a point and $\mathcal{M}$ a degree $6$ line bundle on $C$.
Let $A:=\Pic^5(C)$, and $\mathcal{P}$ be the Poincar\'e bundle on
$C\times A$ normalized so that $\mathcal{P}|_{\{c\}\times A}\simeq \mathcal{O}_A$,
and $\mathcal{E}:=p_{2*}\mathcal{P}$. Let $i:C\to A$ be defined by $i(P):=\mathcal{M}\otimes\mathcal{O}_C(-P)$.

The main goal of this section is to prove the following proposition, thus completing the proof of
Theorem \ref{main} (ii) given in \S \ref{par25}. More precisely, Proposition \ref{ex2} follows from Lemma \ref{iso}, Corollary \ref{coross} and
Proposition \ref{unstable}.

\begin{prop}\label{ex2}
Suppose that $C$ has hyperelliptic equation $Z^2=X^6+Y^6$, and that $\mathcal{M}=\omega_C^{\otimes3}$.
Then $i^*\mathcal{E}$ is strongly semistable if and only if $p\equiv \pm 1[8]$.
\end{prop}

The restrictive assumptions on the curve $C$ and the line bundle $\mathcal{M}$
in the hypotheses of this proposition will be made explicitely later, when they become useful.


\subsection{Syzygy bundles}

Let us first recall what a syzygy bundle is. 

\begin{Def}\label{defsyz}
Let $X$ be a variety, let $(\mathcal{L}_i)_{1\leq i\leq n}$ be line bundles on $X$ and let $\sigma_i\in H^0( X,\mathcal{L}_i)$ be sections 
with no common zero. The syzygy bundle associated to these sections is the vector bundle of rank $n-1$ on $X$ defined by the exact sequence:
\begin{equation}
\label{defsyz}0\to\Syz_X(\sigma_i)\to\bigoplus_i\mathcal{L}_i^{-1}\xrightarrow{\oplus_i \sigma_i}\mathcal{O}_X\to 0.
\end{equation}
\end{Def}
If $\mathcal{N}$ is a line bundle on $X$, one can compute
$H^0(X,\Syz_X(\sigma_i)\otimes\mathcal{N})$ using (\ref{defsyz}): it consists of sections $\tau_i\in H^0(X, \mathcal{L}_i^{-1}\otimes\mathcal{N})$ such that $\sum_i \tau_i\sigma_i=0$.

If $\mathcal{L}$ is a base-point free line bundle on $X$ and the $\sigma_i$ form a base of $H^0(X,\mathcal{L})$, we set $\Syz_X(\mathcal{L}):=\Syz_X(\sigma_i)$.
Let $\mathcal{S}:=\Syz_C(\mathcal{M})$. 

\begin{lem}\label{iso}
There is an isomorphism $i^*\mathcal{E}\simeq \mathcal{S}\otimes\mathcal{M}(c)$.
\end{lem}

\begin{proof}
Consider the pull-back $(\Id,i)^*\mathcal{P}$ of the Poincar\'e bundle on $C\times C$. Its restriction to $\{c\}\times C$ is trivial and its restriction to $C\times\{x\}$
is isomorphic to $\mathcal{M}(-x)$ for every $x\in C(\mathbbm{k})$. It follows that $(\Id,i)^*\mathcal{P}\simeq p_1^*\mathcal{M}\otimes p_2^*\mathcal{O}(c)(-\Delta)$, where $\Delta\subset C\times C$
is the diagonal. As a consequence, there is a short exact sequence on $C\times C$:
$$0\to(\Id,i)^*\mathcal{P}\to p_1^*\mathcal{M}\otimes p_2^*\mathcal{O}(c)\to(p_1^*\mathcal{M}\otimes p_2^*\mathcal{O}(c))|_{\Delta}\to 0.$$
Pushing it forward by $p_2$ and using the vanishing of the appropriate $H^1$, one gets:
$$0\to i^*\mathcal{E}\to H^0(C,\mathcal{M})\otimes \mathcal{O}(c)\to \mathcal{M}(c)\to 0,$$
where the arrow $H^0(C,\mathcal{M})\otimes \mathcal{O}(c)\to \mathcal{M}(c)$ is the evaluation.
One recognizes the definition of a syzygy bundle, up to a twist.
\end{proof}

From now on, we restrict to the case where $\mathcal{M}$ is the tricano\-nical line bundle $\omega_C^{\otimes 3}$. Since
 $\omega_C\simeq f^*\mathcal{O}(1)$, where  $f:C\to \mathbb{P}^1$ is the hyperelliptic double cover, this will allow us to compare
$F^*\mathcal{S}$ with bundles on $\mathbb{P}^1$, that are easier to describe.

Let $X, Y$ be homogeneous coordinates on $\mathbb{P}^1$ and $P(X,Y)$ be a degree $6$ homogeneous polynomial defining the ramification
locus of $f$: the curve $C$ is defined by $Z^2=P(X,Y)$. The canonical ring $\bigoplus_{i\geq 0}H^0(C,\omega_C^{\otimes i})$ of $C$ is then
isomorphic  to $\mathbbm{k}[X,Y,Z]/\langle Z^2-P(X,Y)\rangle$, where the generators $X, Y$ are of degree $1$ and $Z$ is of degree $3$.
In particular, it is isomorphic to $\mathbbm{k}[X,Y]\oplus\mathbbm{k}[X,Y]\cdot Z$ as a $\mathbbm{k}[X,Y]$-module.

Let us introduce the two following syzygy bundles on $\mathbb{P}^1$:
$$\left\{
    \begin{array}{ll}
       \mathcal{S}_+:=\Syz_{\mathbb{P}^1}(X^{3p},X^{2p}Y^p,X^pY^{2p},Y^{3p},P(X,Y)^{\frac{p+1}{2}}), \\
        \mathcal{S}_-:=\Syz_{\mathbb{P}^1}(X^{3p},X^{2p}Y^p,X^pY^{2p},Y^{3p},P(X,Y)^{\frac{p-1}{2}}).
    \end{array}
\right.$$

\begin{lem}\label{exact}
There is an exact sequence:
\begin{equation}
\label{splitsyz}
0\to F^*\mathcal{S}(\omega_C^{\otimes -3})\to f^*\mathcal{S}_+\oplus f^*\mathcal{S}_-(\omega_C^{\otimes -3})\to F^*\mathcal{S}\to 0.
\end{equation}
Moreover, if $m\geq 0$, the complex obtained by tensoring (\ref{splitsyz}) by $\omega_C^{\otimes m}$ and taking
global sections is exact. \end{lem}

\begin{proof}
From the definition of a syzygy bundle, one sees that:
$$F^*\mathcal{S}\simeq\Syz_C(X^{3p}, X^{2p}Y^p,X^pY^{2p},Y^{3p},Z^p).$$

It is easy to describe the morphisms in (\ref{splitsyz}) at the level of local sections.
The morphism $f^*\mathcal{S}_+\to F^*\mathcal{S}$ is $(A,B,C,D,E)\mapsto (A,B,C,D,ZE)$, and the morphism
$f^*\mathcal{S}_-(\omega_C^{\otimes -3})\to F^*\mathcal{S}$ is $(A,B,C,D,E)\mapsto (ZA,ZB,ZC,ZD,E)$. Similarly,
$F^*\mathcal{S}(\omega_C^{\otimes -3})\to f^*\mathcal{S}_+$ is $(A,B,C,D,E)\mapsto (ZA,ZB,ZC,ZD,E)$ and
$F^*\mathcal{S}(\omega_C^{\otimes -3})\to f^*\mathcal{S}_-(\omega_C^{\otimes -3})$ is $(A,B,C,D,E)\mapsto -(A,B,C,D,ZE)$.

To prove the exactness of (\ref{splitsyz}), it suffices to prove the second statement.
This is easy using the description of the canonical ring as $\mathbbm{k}[X,Y]\oplus\mathbbm{k}[X,Y]\cdot Z$.
\end{proof}

\subsection{The strong Lefschetz property and syzygy computations}\label{parsyz}

To compute
the syzygy bundles $\mathcal{S}_+$ and $\mathcal{S}_-$, we need to restrict the situation again, by choosing carefully the polynomial $P$.
We will take $P(X,Y)= X^6+Y^6$, so that $C$ is the curve of equation $Z^2=X^6+Y^6$. In particular, from now on, we suppose that $p\neq 2,3$
so that $C$ is indeed smooth.

Our main tool will be the strong Lefschetz property for homogeneous ideals.

\begin{Def}
Let $R:= \mathbbm{k} [x_1,\dots,x_n]$. A homogeneous Artinian ideal
$I \subset R$ satisfies the \textit{strong Lefschetz property}
if there is a linear form $l\in R_1$ such that for every $r,d\geq 0$, the multiplication map
$(R/I)_r\stackrel{\cdot l^d}{\longrightarrow}(R/I)_{r+d}$ is of maximal rank.
\end{Def}

\begin{lem}\label{Lefschetz}
Let $I\subset  \mathbbm{k} [x,y]$ be a homogeneous Artinian ideal. 
Suppose that $(R/I)_r=0$ for $r\geq p$.
Then $I$ satisfies the strong Lefschetz property.
\end{lem}

\begin{proof}
In characteristic $0$, this is \cite[Proposition 4.4]{Lefschetz}. 
In this proof, the characteristic $0$ hypothesis is only used for the explicit description of Borel-fixed ideals, applied to the generic initial ideal of $I$. The description of Borel-fixed ideals in positive cha\-racteristic $p$ \cite[Theorem 15.23]{Eisenbud} is
more complicated in general, but coincides with the simple one in characteristic $0$
when the condition that $(R/I)_r=0$ for $r\geq p$ is satisfied. Consequently, under this hypothesis, the proof goes through.
\end{proof}

It is now possible to prove:
\begin{prop} \label{syzcalcul}\hspace{1em}
\begin{enumerate}[(i)]
\item 
If $p\equiv 1[8]$, $\mathcal{S}_+\simeq \mathcal{O}(\frac{-15p-9}{4})\oplus\mathcal{O}(\frac{-15p-1}{4})^{\oplus 3}$ and  $\mathcal{S}_-\simeq \mathcal{O}(\frac{-15p+3}{4})^{\oplus 4}$.
\item If $p\equiv -1[8]$, $\mathcal{S}_+\simeq \mathcal{O}(\frac{-15p-3}{4})^{\oplus 4}$ and  $\mathcal{S}_-\simeq \mathcal{O}(\frac{-15p+1}{4})^{\oplus 3}\oplus\mathcal{O}(\frac{-15p+9}{4})$.
\item If $p\equiv 3[8]$,  $\mathcal{S}_+\simeq \mathcal{O}(\frac{-15p-7}{4})^{\oplus 2}\oplus\mathcal{O}(\frac{-15p+1}{4})^{\oplus 2}$ and  $\mathcal{S}_-\simeq  \mathcal{O}(\frac{-15p-3}{4})\oplus\mathcal{O}(\frac{-15p+5}{4})^{\oplus 3}$.
\item If $p\equiv -3[8]$,  $\mathcal{S}_+\simeq \mathcal{O}(\frac{-15p-5}{4})^{\oplus 3}\oplus\mathcal{O}(\frac{-15p+3}{4})$ and  $\mathcal{S}_-\simeq  \mathcal{O}(\frac{-15p-1}{4})^{\oplus 2}\oplus\mathcal{O}(\frac{-15p+7}{4})^{\oplus 2}$.
\end{enumerate}
\end{prop}

\begin{proof}
We know the degrees of $\mathcal{S}_+$ and $\mathcal{S}_-$ from their definition. Moreover, by Grothendieck's theorem,
a vector bundle on $\mathbb{P}^1$ splits as a direct sum of line bundles. As a consequence, to prove the proposition, it is enough to compute
the global sections of some twists of $\mathcal{S}_+$ and $\mathcal{S}_-$. For instance, to prove that 
 $\mathcal{S}_+\simeq \mathcal{O}(\frac{-15p-9}{4})\oplus\mathcal{O}(\frac{-15p-1}{4})^{\oplus 3}$ if $p\equiv 1[8]$,
it is sufficient to show that $h^0(\mathbb{P}^1,\mathcal{S}_+(\frac{15p-3}{4}))=0$ and that $h^0(\mathbb{P}^1,\mathcal{S}_+(\frac{15p+1}{4}))=3$.

Even if the result depends only on $p$ modulo $8$, we distinguish between different values of $p$ modulo $24$.
As all the global section computations needed are similar, we only carry out one: assuming that $p\equiv 1[24]$, we prove that $h^0(\mathbb{P}^1,\mathcal{S}_+(\frac{15p+1}{4}))=3$. 

Applying global sections to an appropriate twist of the exact sequence defining $\mathcal{S}_+$,
we see that $H^0(\mathbb{P}^1,\mathcal{S}_+(\frac{15p+1}{4}))$ is the vector space of solutions of the equation:
\begin{equation}\label{system1}
AX^{3p}+BX^{2p}Y^p+CX^pY^{2p}+DY^{3p}+ E(X^6+Y^6)^{\frac{p+1}{2}}=0,
\end{equation}
where the unknowns $A,B,C,D,E$ are homogeneous polynomials in $X$ and $Y$, the first four being of degree $\frac{3p+1}{4}$
and $E$ being of degree $\frac{3p-11}{4}$. Equation (\ref{system1}) is a linear system in the coefficients
of $A,B, C,D,E$. 

The matrix of this linear system in the monomial bases has six rectangular blocks, as one sees by separating the monomials
according to the value modulo 6 of the exponent of $X$. Consequently, the solution space of (\ref{system1}) is the direct sum of the solution
spaces of six smaller systems, that we may solve independently. 

Let us look at the first one, obtained by keping in (\ref{system1}) only monomials in which the exponent of $X$ is a multiple of $6$.
Then, setting $x:=X^6$ and $y:=Y^6$, it is possible to write $A=X^3Y^4a(x,y)$, $B=X^4Y^3b(x,y)$, $C=X^5Y^2c(x,y)$, $D=Yd(x,y)$
and $E=Y^4e(x,y)$. Dividing by $Y^4$, we get the new equation:
\begin{equation}\label{system2}
ax^{\frac{p+1}{2}}+bx^{\frac{p+2}{3}}y^{\frac{p-1}{6}}+cx^{\frac{p+5}{6}}y^{\frac{p-1}{3}}+dy^{\frac{p-1}{2}}+ e(x+y)^{\frac{p+1}{2}}=0,
\end{equation}
where the unknowns $a,b,c,d,e$ are homogeneous polynomials in $x$ and $y$ of respective degrees
$\frac{p-9}{8}$, $\frac{p-9}{8}$, $\frac{p-9}{8}$, $\frac{p-1}{8}$ and $\frac{p-9}{8}$: it is a linear system in $\frac{5p+3}{8}$ unknowns and
as many equations.

Introduce the ideal $I:=\langle  x^{\frac{p+1}{2}}, x^{\frac{p+2}{3}}y^{\frac{p-1}{6}}, x^{\frac{p+5}{6}}y^{\frac{p-1}{3}}, y^{\frac{p-1}{2}}\rangle$ of $R:=\mathbbm{k}[x,y]$. The linear system (\ref{system2}) has maximal rank exactly when
$\cdot (x+y)^{\frac{p+1}{2}}:(R/I)_{\frac{p-9}{8}}\to(R/I)_{\frac{5p-5}{8}}$ has maximal rank. If $\alpha,\beta\in\mathbbm{k}^*$, this rank is equal to the rank of
multiplication by $(\alpha x+\beta y)^{\frac{p+1}{2}}$,
as one sees by performing the change of variables $x'=\alpha x$, $y'=\beta y$,
hence of the multiplication by a power of a general linear form.
By Lemma \ref{Lefschetz}, $I$ satisfies the strong Lefschetz property and such
multiplication maps have maximal rank. We have proven that (\ref{system2}) has maximal rank. Since it has as many unknowns as equations,
it has no nontrivial solution.

The same argument using the strong Lefschetz property shows that the five other sub-linear systems have maximal rank. Three of them (corresponding to exponents of $X$ congruent to $1$, $2$ and $3$ modulo $6$) have
exactly one more
unknown than equations. Another has as many unknowns as
equations (the one corresponding to exponents of $X$ congruent to $4$ modulo $6$), and the last one has more equations than unknowns. 
Consequently, only three have non-trivial solutions, and moreover a one-dimensional solution space. It follows, as wanted, that
 $h^0(\mathbb{P}^1,\mathcal{S}_+(\frac{15p+1}{4}))=3$.
\end{proof}

\begin{rem}
The matrices of the linear systems in the proof of Proposition \ref{syzcalcul} are complicated matrices of binomial coefficients,
very similar to those appearing in Han's thesis \cite{Han}.
It seems difficult to check directly that they are of maximal rank.
\end{rem}

\begin{rem}
Proposition \ref{syzcalcul} and Lemma \ref{exact} show at once that $F^*\mathcal{S}$ is unstable when $p\equiv \pm 3[8]$. We
will obtain more precise information in Paragraph \ref{parunst}.
\end{rem}

\subsection{Frobenius periodicity}

We are ready to prove the strong semistability of $\mathcal{S}$ when $p\equiv\pm 1[8]$. Denote by $R$ the ramification locus of $f$: it consists
of the points $P_i=[\zeta_i:1]$, where the $\zeta_i$ are the sixth roots of $-1$. We view $R$ either as a subset of $\mathbb{P}^1$
or as a subset of $C$. Note that these ramification points are
transitively permuted by the natural action of the group $\mu_6$ of sixth roots of unity on $\mathbb{P}^1$.

\begin{prop}
\label{propexact} 
There are exact sequences: 
 \begin{equation}
\label{exa1}
0\to F^*\mathcal{S}\to(\omega_C^{\otimes\frac{-15p+3}{4}})^{\oplus 5}\to\omega_C^{\otimes\frac{-15p+15}{4}}\to 0
\textrm{, if }p\equiv 1[8],
\end{equation}
 \begin{equation}
\label{exa2}
0\to\omega_C^{\otimes\frac{-15p-15}{4}}\to(\omega_C^{\otimes\frac{-15p-3}{4}})^{\oplus 5}\to F^*\mathcal{S}\to 0
\textrm{, if }p\equiv -1[8].
\end{equation}
\end{prop}

\begin{proof}
We first construct (\ref{exa1}).
By Proposition \ref{syzcalcul},
the injective morphism in (\ref{splitsyz}) writes:
$F^*\mathcal{S}\to (\omega_C^{\otimes\frac{-15p+11}{4}})^{\oplus 3}\oplus(\omega_C^{\otimes\frac{-15p+3}{4}})\oplus (\omega_C^{\otimes\frac{-15p+3}{4}})^{\oplus 4}$. We will prove that the induced morphism $F^*\mathcal{S}\to (\omega_C^{\otimes\frac{-15p+3}{4}})^{\oplus 5}$ is injective in restriction to every point $P\in C$. This concludes because its quotient is then a line bundle,
isomorphic to $\omega_C^{\otimes\frac{-15p+15}{4}}$ for degree reasons.

 From its description, one sees that the morphism
$F^*\mathcal{S}\to f^*\mathcal{S}_-\simeq(\omega_C^{\otimes\frac{-15p+3}{4}})^{\oplus 4}$ is an isomorphism on the fibers outside
$R$, and that if $P\in R$, the kernel of $F^*\mathcal{S}|_P\to f^*\mathcal{S}_-|_P$
consists of syzygies $(A,B,C,D,E)$ such that $A,B,C,D$ vanish at $P$. It remains to see that this kernel is not killed by the composition
$F^*\mathcal{S}|_P\to f^*\mathcal{S}_+(\omega_C^{\otimes 3})|_P\to \omega_C^{\otimes\frac{-15p+3}{4}}|_P$ or equivalently that its image in
$f^*\mathcal{S}_+(\omega_C^{\otimes 3})|_P$
does not belong to $(\omega_C^{\otimes\frac{-15p+11}{4}})^{\oplus 3}|_P$. 

Suppose that it is not the case for $P=P_1$: then there exists a non-zero section
$(A,B,C,D,E)\in H^0(C, f^*\mathcal{S}_+(\omega_C^{\otimes\frac{15p+1}{4}}))
=H^0(\mathbb{P}^1, \mathcal{S}_+(\frac{15p+1}{4}))$ such that $A,B,C,D$ vanish at $ P_1$. Writing $A=(X-\zeta_1Y)\tilde{A}_1$,
$B=(X-\zeta_1Y)\tilde{B}_1$, $C=(X-\zeta_1Y)\tilde{C}_1$, $D=(X-\zeta_1Y)\tilde{D}_1$,
$\tilde{E}_1=\prod_{i=2}^6(X-\zeta_iY)E$, one gets a section
$\sigma_1=(\tilde{A}_1,\tilde{B}_1,\tilde{C}_1,\tilde{D}_1,\tilde {E}_1)\in
H^0(\mathbb{P}^1, \mathcal{S}_-(\frac{15p-3}{4}))$ such that $\tilde{E}_1$ vanishes at $P_2,\dots,P_6$. For symmetry reasons,
using the $\mu_6$-action, there exists,
for every $1\leq i\leq 6$ a non-zero section $\sigma_i=(\tilde{A}_i,\tilde{B}_i,\tilde{C}_i,\tilde{D}_i,\tilde {E}_i)
\in H^0(\mathbb{P}^1, \mathcal{S}_-(\frac{15p-3}{4}))$ such that $\tilde{E}_i$ vanishes at $P_j$ for $j\neq i$. Since
 $H^0(\mathbb{P}^1, \mathcal{S}_-(\frac{15p-3}{4}))$ is $4$-dimensional by Proposition \ref{syzcalcul}, these six
sections cannot be linearly independent:
for instance, $\sigma_1\in\langle\sigma_2,\dots,\sigma_6\rangle$. It follows that $\tilde{E}_1$ vanishes at all $P_i$.
Then $(\tilde{A}_1,\tilde{B}_1,\tilde{C}_1,\tilde{D}_1,\tilde{E}_1/(X^6+Y^6))\in H^0(\mathbb{P}^1, \mathcal{S}_+(\frac{15p-3}{4}))$ is non-zero, contradicting
Proposition \ref{syzcalcul}.  

Let us explain how to obtain (\ref{exa2}) by a similar argument. By Lemma \ref{exact} and Proposition \ref{syzcalcul},
there is a morphism $(\omega_C^{\otimes\frac{-15p-3}{4}})^{\oplus 5}\to F^*\mathcal{S}$, and it suffices to prove its surjectivity.
Using only the four factors coming from $\mathcal{S}_+$, one gets surjectivity at points $P\notin R$, and the fact that, if $P\in R$, all
$(A,B,C,D,E)\in  F^*\mathcal{S}|_P$ such that $E=0$  are in the image. Hence, it suffices to prove that the unique
section $(A,B,C,D,E)\in H^0(C,f^*\mathcal{S}_-(\omega_C^{\otimes\frac{15p-9}{4}}))
=H^0(\mathbb{P}^1,\mathcal{S}_-(\frac{15p-9}{4}))$ satisfies $E(P)\neq 0$. If it didn't, $E_R$ would vanish by symmetry, and
 $(A,B,C,D,E/(X^6+Y^6))\in H^0(\mathbb{P}^1,\mathcal{S}_+(\frac{15p-9}{4}))$ would be a non-zero section contradicting Proposition
\ref{syzcalcul}.
\end{proof}

\begin{prop}\label{Frobper} There are isomorphisms:
\begin{enumerate}[(i)]
\item $F^*\mathcal{S}\simeq \mathcal{S}(\omega_C^{\otimes\frac{15-15p}{4}})$, if $p\equiv 1[8]$,
\item $F^*\mathcal{S}\simeq \mathcal{S}^\vee(\omega_C^{\otimes\frac{-15-15p}{4}})$, if $p\equiv -1[8]$.
\end{enumerate}
\end{prop}

\begin{proof}
 Denote by $\sigma_i\in H^0(C,\omega_C^{\otimes 3})$ the sections appearing in the last arrow
of (\ref{exa1}).
Tensoring (\ref{exa1}) by $\omega_C^{\otimes\frac{15p-3}{4}}$ and taking cohomology, one gets:
$$0\to H^0(C,F^*\mathcal{S}(\omega_C^{\otimes\frac{15p-3}{4}}))\to\mathbbm{k}^{\oplus 5}
\stackrel{\sigma_i}{\longrightarrow} H^0(C,\omega_C^{\otimes 3}).$$
But $ H^0(C,F^*\mathcal{S}(\omega_C^{\otimes\frac{15p-3}{4}}))=0$ by the second part of Lemma \ref{exact} applied with $m=\frac{15p-3}{4}$
and Proposition \ref{syzcalcul}. Thus, the $\sigma_i$ are linearly independant and (\ref{exa1}) is, up to a twist, the exact sequence defining the syzygy bundle $\mathcal{S}$, proving (i).

We prove (ii) in a similar way. Denote by $\tau_i\in H^0(C,\omega_C^{\otimes 3})$
the sections appearing in the first arrow of (\ref{exa2}). Tensoring it by $\omega_C^{\otimes\frac{15p+7}{4}}$
and taking cohomology, one gets:
$$0\to H^0(C,\omega_C)^{\oplus 5}\to H^0(C, F^*\mathcal{S}(\omega_C^{\otimes\frac{15p+7}{4}}))\to 
H^1(C, \omega_C^{\otimes -2})\stackrel{\tau_i}{\longrightarrow} H^1(C,\omega_C)^{\oplus 5}.$$
The vector space $H^0(C,\omega_C)^{\oplus 5}$ is $10$-dimensional. By the second part
of Lemma \ref{exact} and Proposition \ref{syzcalcul}, one sees that $H^0(C, F^*\mathcal{S}(\omega_C^{\otimes\frac{15p+7}{4}}))$ has dimension $\leq 10$.
It follows that $H^1(C, \omega_C^{\otimes -2}))\to H^1(C,\omega_C)^{\oplus 5}$
is injective. This map being Serre-dual to $\mathbbm{k}^{\oplus 5}
\stackrel{\tau_i}{\longrightarrow} H^0(C,\omega_C^{\otimes 3})$, the $\tau_i$ generate $H^0(C,\omega_C^{\otimes 3})$. Hence,
the dual of (\ref{exa2}) is, up to a twist, the exact sequence defining the syzygy bundle $\mathcal{S}$, proving (ii).
\end{proof}

\begin{cor}\label{coross}
If $p\equiv\pm 1[8]$, $\mathcal{S}$ is strongly semistable. 
\end{cor}

\begin{proof}
Proposition \ref{Frobper} shows that when $p\equiv\pm 1[8]$, 
$\mathcal{S}$ is Frobenius periodic up to a twist: $F^{2*}\mathcal{S}\simeq\mathcal{S}(\omega_C^{\otimes\frac{15p^2-15}{4}})$. It is classical that such bundles are strongly semistable. We recall the argument.
Suppose that $\mathcal{S}$ is not semistable, and let $\mathcal{F}\subset\mathcal{S}$ be the first graded piece of its Harder-Narasimhan filtration. 
Then $F^{2*}\mathcal{F}(\omega_C^{\otimes\frac{-15p^2+15}{4}})\subset \mathcal{S}$
has greater slope than $\mathcal{F}$, a contradiction. Hence $\mathcal{S}$ is semistable. By the periodicity property, so are all its Frobenius pull-backs.
\end{proof}

\subsection{Unstability}\label{parunst}

Let us now describe what happens when $p\equiv \pm3 [8]$.

\begin{prop}\label{unstable}
If $p\equiv \pm 3 [8]$, then $F^*\mathcal{S}$ is not semistable and its Harder-Narasimhan filtration is strong. This filtration is of the form:
\begin{enumerate}[(i)]
\item  $0\to \mathcal{T}\to F^*\mathcal{S}\to \omega_C^{\otimes\frac{-15p-3}{4}}\to 0$  if $p\equiv 3[8]$,
\item $0\to \omega_C^{\otimes\frac{-15p+3}{4}}\to F^*\mathcal{S}\to \mathcal{T} \to 0$ if $p\equiv -3[8]$.
\end{enumerate}
\end{prop}

\begin{proof} We will only prove (i), as the second statement is similar. From Lemma \ref{exact} and Proposition \ref{syzcalcul}, we get a morphism
$F^*\mathcal{S}\to f^*\mathcal{S}_-\to\omega_C^{\otimes\frac{-15p-3}{4}}$. Let us prove that it is surjective. Since
$F^*\mathcal{S}\to f^*\mathcal{S}_-$ is surjective at all points $P\notin R$, and since if $P\in R$,
the image of $F^*\mathcal{S}|_P\to f^*\mathcal{S}_-|_P$
consists of syzygies $(A,B,C,D,E)$ such that $E(P)=0$, we need to show that not all sections
$(A,B,C,D,E)\in H^0(C, f^*\mathcal{S}_-(\omega_C^{\otimes\frac{15p-5}{4}}))=H^0(\mathbb{P}^1,\mathcal{S}_-(\frac{15p-5}{4}))$ satisfy $E(P)=0$.
Suppose it is not the case: then, by symmetry using the $\mu_6$-action,
for all sections $(A,B,C,D,E)\in H^0(\mathbb{P}^1,\mathcal{S}_-(\frac{15p-5}{4}))$, $E$ would 
vanish on $R$. Dividing $E$ by $X^6+Y^6$, we would get a non-sero section in $H^0(\mathbb{P}^1,\mathcal{S}_+(\frac{15p-5}{4}))$, contradicting
Proposition \ref{syzcalcul}. Hence our morphism was surjective, and we denote its kernel by $\mathcal{T}$.

From Lemma \ref{exact} and Proposition \ref{syzcalcul} again, we get a morphism $(\omega_C^{\otimes\frac{-15p+1}{4}})^{\oplus 2}
\to f^*\mathcal{S}_+\to F^*\mathcal{S}$. Let us prove that it is injective on every fiber. Since $f^*\mathcal{S}_+\to F^*\mathcal{S}$ is injective
on the fibers at $P\notin R$, and since, if $P\in R$, the kernel of $f^*\mathcal{S}_+|_P\to F^*\mathcal{S}|_P$ consists of syzygies
$(A,B,C,D,E)$ such that $A,B,C,D$ all vanish at $P$, it suffices to rule out the existence of a section $(A,B,C,D,E)\in H^0(C,f^*\mathcal{S}_+(\omega_C^{\otimes\frac{15p-1}{4}}))= H^0(\mathbb{P}^1,\mathcal{S}_+(\frac{15p-1}{4}))$
such that $A,B,C,D$ all vanish at $P$. We proceed by contradiction. Then, for symmetry reasons, there exist for $1\leq i\leq 6$ a section
$(A_i,B_i,C_i,D_i,E_i)\in H^0(\mathbb{P}^1,\mathcal{S}_+(\frac{15p-1}{4}))$ such that $A_i,B_i,C_i,D_i$ all vanish at $P_i$.
Dividing $A_i$, $B_i$, $C_i$, $D_i$ by $X-\zeta_i Y$ and multiplying $E$ by $\prod_{j\neq i}(X-\zeta_j Y)$, we get non-zero sections
$\sigma_i=(\tilde{A}_i,\tilde{B}_i,\tilde{C}_i,\tilde{D}_i,\tilde {E}_i)\in H^0(\mathbb{P}^1,\mathcal{S}_-(\frac{15p-5}{4}))$ such that 
$\tilde{E}_i$ vanishes at $P_j$ for $j\neq i$. By Proposition \ref{syzcalcul}, $H^0(\mathbb{P}^1,\mathcal{S}_-(\frac{15p-5}{4}))$
is $3$-dimensional, hence the $\sigma _i$ cannot be linearly independent, say $\sigma_1\in\langle \sigma_2,\dots, \sigma_6\rangle$.
Then $\tilde{E}_1$ vanishes at all the $P_i$ and $(\tilde{A}_1,\tilde{B}_1,\tilde{C}_1,\tilde{D}_1,\tilde {E}_1/(X^6+Y^6))\in
H^0(\mathbb{P}^1,\mathcal{S}_+(\frac{15p-5}{4}))$ is a non-zero section contradicting Proposition \ref{syzcalcul}.

Since there are obviously no non-zero morphisms $\omega_C^{\otimes\frac{-15p+1}{4}}\to\omega_C^{\otimes\frac{-15p-3}{4}}$, the subbundle
$(\omega_C^{\otimes\frac{-15p+1}{4}})^{\oplus 2}$ factors through $\mathcal{T}$, and a degree computation shows that
this realizes $\mathcal{T}$ as an extension:
\begin{equation}\label{extension}0\to(\omega_C^{\otimes\frac{-15p+1}{4}})^{\oplus 2}\to\mathcal{T}\to\omega_C^{\otimes\frac{-15p+1}{4}}\to 0.
\end{equation}
Now $\mathcal{T}$ is strongly semistable as an extension of strongly semistable bundles of the same slope, and writing $F^*\mathcal{S}$ as
an extension of $\omega_C^{\otimes\frac{-15p-3}{4}}$ by $\mathcal{T}$ indeed realizes the Harder-Narasimhan filtration of $F^*\mathcal{S}$.
\end{proof}

\subsection{Hilbert-Kunz multiplicities}\label{HK}
We now apply our results to the computation of Hilbert-Kunz multiplicities.
Let us first recall the definition.

\begin{Def}\label{DefHK}
Let $A$ be a noetherian $n$-dimensional ring of characteristic $p$ and $\mathfrak{m}$ be a maximal ideal of $A$. Let $\mathfrak{m}^{[e]}$ be the ideal of $A$ generated by $p^e$-th powers of elements of $\mathfrak{m}$.
The \textit{Hilbert-Kunz multiplicity} of $(A,\mathfrak{m})$
is defined to be: $$e_{\HK}(A,\mathfrak{m}):=\lim_{e\to\infty} \frac{\length(A/\mathfrak{m}^{[e]})}{p^{ne}}.$$
\end{Def}

This invariant was first considered by Kunz \cite{Kunz}, and the limit was shown to exist and to be finite by Monsky \cite{HilbertKunz}.
It is difficult to compute in general.

We will be interested in the following geometric case:

\begin{Def}
\label{defHK}
Let $C$ be a smooth curve endowed with a line bundle $\mathcal{L}$ whose sections embed $C$ as a projectively normal curve.
Consider the section ring $A:=\bigoplus_{l\geq 0}H^0(C,\mathcal{L}^{\otimes l})$ with its maximal ideal
$\mathfrak{m}:=\bigoplus_{l>0}H^0(C,\mathcal{L}^{\otimes l})$.
Define: $$e_{\HK}(C,\mathcal{L}):=e_{\HK}(A,\mathfrak{m}).$$
\end{Def}
In this particular case,  Brenner \cite[Theorem 1]{Brenner} and Trivedi \cite[Theorem 4.12]{Trivedi} have related
the Hilbert-Kunz multiplicity to properties of a syzygy bundle:

\begin{thm}[Brenner, Trivedi]
\label{BrennerTrivedi}
Let $C$ be a smooth curve endowed with a degree $d$ line bundle $\mathcal{L}$ whose sections embed $C$ in
$\mathbb{P}^{k-1}$ as a projectively normal curve.
Using Theorem \ref{HNfort}, choose a finite morphism of degree $e$ from a smooth curve $f:C'\to C$ such that the Harder-Narasimhan filtration
 of $f^{*}\Syz_C(\mathcal{L})$ is strong.
Let $r_i$ and $e\delta_i$ be the ranks and degrees of the graded pieces of this filtration (so that $r_i$ and $\delta_i$ are independent of $f$).
Then:
$$e_{\HK}(C,\mathcal{L})=\frac{1}{2d}\sum_i\frac{\delta_i^2}{r_i}-\frac{kd}{2}.$$
\end{thm}

Applying this theorem using Corollary \ref{coross} and Proposition \ref{unstable}, we get:

\begin{thm}
\label{thHK}
Let $C$ be the curve of genus $2$ with equation $Z^2=X^6+Y^6$. Then:
\begin{enumerate}[(i)]
\item $e_{\HK}(C,\omega_C^{\otimes 3})=\frac{15}{4}$ if $p\equiv\pm 1[8]$,
\item $e_{\HK}(C,\omega_C^{\otimes 3})=\frac{15}{4}+\frac{1}{4p^2}$ if $p\equiv\pm 3[8]$.
\end{enumerate}
\end{thm}


\bibliographystyle{plain}
\bibliography{biblio}

\begin{thebibliography}{10}

\bibitem{ACGH1}
E.~Arbarello, M.~Cornalba, P.~A. Griffiths, and J.~Harris.
\newblock {\em Geometry of algebraic curves. {V}ol. {I}}, volume 267 of {\em
  Grundlehren der Mathematischen Wissenschaften}.
\newblock Springer-Verlag, New York, 1985.

\bibitem{Oolbir}
O.~Benoist.
\newblock On the birational geometry of the parameter space for codimension 2
  complete intersections.
\newblock {\em arXiv:1212.4862}, 2012.

\bibitem{Brenner}
H.~Brenner.
\newblock The rationality of the {H}ilbert-{K}unz multiplicity in graded
  dimension two.
\newblock {\em Math. Ann.}, 334(1):91--110, 2006.

\bibitem{BrennerKaid}
H.~Brenner and A.~Kaid.
\newblock An explicit example of {F}robenius periodicity.
\newblock {\em J. Pure Appl. Algebra}, 217(8):1412--1420, 2013.

\bibitem{ChangRan1}
M.-C. Chang and Z.~Ran.
\newblock Closed families of smooth space curves.
\newblock {\em Duke Math. J.}, 52(3):707--713, 1985.

\bibitem{ChangRan2}
M.-C. Chang and Z.~Ran.
\newblock Dimension of families of space curves.
\newblock {\em Compositio Math.}, 90(1):53--57, 1994.

\bibitem{Dale}
M.~Dale.
\newblock Terracini's lemma and the secant variety of a curve.
\newblock {\em Proc. London Math. Soc. (3)}, 49(2):329--339, 1984.

\bibitem{Eisenbud}
D.~Eisenbud.
\newblock {\em Commutative algebra with a view toward algebraic geometry},
  volume 150 of {\em Graduate Texts in Mathematics}.
\newblock Springer-Verlag, New York, 1995.

\bibitem{FreyKani}
G.~Frey and E.~Kani.
\newblock Curves of genus {$2$} covering elliptic curves and an arithmetical
  application.
\newblock In {\em Arithmetic algebraic geometry ({T}exel, 1989)}, volume~89 of
  {\em Progr. Math.}, pages 153--176. Birkh\"auser Boston, Boston, MA, 1991.

\bibitem{Han}
C.~Han.
\newblock {\em The {H}ilbert-{K}unz function of a diagonal hypersurface}.
\newblock ProQuest LLC, Ann Arbor, MI, 1992.
\newblock Thesis (Ph.D.)--Brandeis University.

\bibitem{HanMonsky}
C.~Han and P.~Monsky.
\newblock Some surprising {H}ilbert-{K}unz functions.
\newblock {\em Math. Z.}, 214(1):119--135, 1993.

\bibitem{Lefschetz}
T.~Harima, J.~C. Migliore, U.~Nagel, and J.~Watanabe.
\newblock The weak and strong {L}efschetz properties for {A}rtinian
  {$K$}-algebras.
\newblock {\em J. Algebra}, 262(1):99--126, 2003.

\bibitem{HM}
J.~Harris and I.~Morrison.
\newblock {\em Moduli of curves}, volume 187 of {\em Graduate Texts in
  Mathematics}.
\newblock Springer-Verlag, New York, 1998.

\bibitem{Hartshorne}
R.~Hartshorne.
\newblock {\em Algebraic geometry}, volume~52 of {\em Graduate Texts in
  Mathematics}.
\newblock Springer-Verlag, New York, 1977.

\bibitem{Ploog}
G.~Hein and D.~Ploog.
\newblock Fourier-{M}ukai transforms and stable bundles on elliptic curves.
\newblock {\em Beitr\"age Algebra Geom.}, 46(2):423--434, 2005.

\bibitem{Huybrechts}
D.~Huybrechts and M.~Lehn.
\newblock {\em The geometry of moduli spaces of sheaves}.
\newblock Cambridge Mathematical Library. Cambridge University Press,
  Cambridge, second edition, 2010.

\bibitem{Igusa}
J.~Igusa.
\newblock Arithmetic variety of moduli for genus two.
\newblock {\em Ann. of Math. (2)}, 72:612--649, 1960.

\bibitem{descinsep}
N.~M. Katz.
\newblock Nilpotent connections and the monodromy theorem: {A}pplications of a
  result of {T}urrittin.
\newblock {\em Inst. Hautes \'Etudes Sci. Publ. Math.}, (39):175--232, 1970.

\bibitem{Kunz}
E.~Kunz.
\newblock Characterizations of regular local rings for characteristic {$p$}.
\newblock {\em Amer. J. Math.}, 91:772--784, 1969.

\bibitem{LangeStuhler}
H.~Lange and U.~Stuhler.
\newblock Vektorb\"undel auf {K}urven und {D}arstellungen der algebraischen
  {F}undamentalgruppe.
\newblock {\em Math. Z.}, 156(1):73--83, 1977.

\bibitem{Langer}
A.~Langer.
\newblock Semistable sheaves in positive characteristic.
\newblock {\em Ann. of Math. (2)}, 159(1):251--276, 2004.

\bibitem{Langersurvey}
A.~Langer.
\newblock Moduli spaces of sheaves and principal {$G$}-bundles.
\newblock In {\em Algebraic geometry---{S}eattle 2005. {P}art 1}, volume~80 of
  {\em Proc. Sympos. Pure Math.}, pages 273--308. Amer. Math. Soc., Providence,
  RI, 2009.

\bibitem{MehtaRama}
V.~B. Mehta and A.~Ramanathan.
\newblock Homogeneous bundles in characteristic {$p$}.
\newblock In {\em Algebraic geometry---open problems ({R}avello, 1982)}, volume
  997 of {\em Lecture Notes in Math.}, pages 315--320. Springer, Berlin, 1983.

\bibitem{HilbertKunz}
P.~Monsky.
\newblock The {H}ilbert-{K}unz function.
\newblock {\em Math. Ann.}, 263(1):43--49, 1983.

\bibitem{Monsky}
P.~Monsky.
\newblock The {H}ilbert-{K}unz multiplicity of an irreducible trinomial.
\newblock {\em J. Algebra}, 304(2):1101--1107, 2006.

\bibitem{Oda}
T.~Oda.
\newblock Vector bundles on an elliptic curve.
\newblock {\em Nagoya Math. J.}, 43:41--72, 1971.

\bibitem{Oortcomplete}
F.~Oort.
\newblock Subvarieties of moduli spaces.
\newblock {\em Invent. Math.}, 24:95--119, 1974.

\bibitem{Schmid}
H.~L. Schmid.
\newblock \"{U}ber die {A}utomorphismen eines algebraischen
  {F}unktionenk\"orpers von {P}rimzahlcharakteristik.
\newblock {\em J. Reine Angew. Math.}, 179:5--15, 1938.

\bibitem{Szpiro}
L.~{Szpiro}.
\newblock {Propri\'et\'es num\'eriques du faisceau dualisant relatif.}
\newblock {\em {Ast\'erisque}}, 86:44--78, 1981.

\bibitem{Trivedi}
V.~Trivedi.
\newblock Semistability and {H}ilbert-{K}unz multiplicities for curves.
\newblock {\em J. Algebra}, 284(2):627--644, 2005.

\end{thebibliography}

\end{document}